\documentclass[12pt]{article}
\usepackage{amsfonts}

\usepackage{textalpha} 

\usepackage{bm}
\usepackage{enumerate} 
\usepackage{amssymb,amsmath}
\usepackage{amsthm} 
\usepackage{caption}
\usepackage{subcaption}
\usepackage{accents} 

\usepackage{dsfont}

\usepackage{stackengine}

\usepackage{accents}

\usepackage{tikz}
\usetikzlibrary{automata,topaths}
\usetikzlibrary{shapes}
\usetikzlibrary{plotmarks}
 
\usepackage{float} 
\usepackage{textcomp} 
\usepackage{siunitx}
\usepackage{color} 
\definecolor{lightblue}{rgb}{0,0.2,0.5}
\usepackage[colorlinks=true, urlcolor=lightblue,linkcolor=lightblue, citecolor=lightblue]{hyperref}
\usepackage{ucs}

\definecolor{dblackcolor}{rgb}{0.0,0.0,0.0}
\definecolor{dbluecolor}{rgb}{0.01,0.02,0.7}
\definecolor{dgreencolor}{rgb}{0.2,0.4,0.0}
\definecolor{dgraycolor}{rgb}{0.30,0.3,0.30}

\usepackage{xcolor}
\definecolor{ForestGreen}{RGB}{34,139,34}
\definecolor{mauve}{rgb}{0.7,0,0.43}
\definecolor{dkgreen}{rgb}{0,0.6,0}
\definecolor{darkgreen}{rgb}{0,0.6,0}
\definecolor{darkorange}{rgb}{1.0, 0.55, 0.0}
\definecolor{lightblue}{rgb}{0,0.2,0.5}
\definecolor{blue1}{rgb}{0,0.1,0.9}

\definecolor{codegreen}{rgb}{0,0.6,0}
\definecolor{codegray}{rgb}{0.5,0.5,0.5}
\definecolor{backcolour}{rgb}{0.97,0.97,0.97}

\definecolor{lightblue}{rgb}{0,0.2,0.5}

\usepackage{listings}
\lstdefinelanguage{Maple}{
    morekeywords={proc, if, return, map, op, int, for, do, local, nops, convert, end},
    sensitive=false, 
    morecomment=[l]{//}, 
    morecomment=[s]{/*}{*/}, 
    morestring=[b]" 
} 

\usepackage{graphicx}
\usepackage{flushend,cuted}
\usepackage{bm}
\usepackage{tabularx}

\usepackage{indentfirst}
\usepackage{amssymb}
\usepackage{xparse}
\usepackage{tikz}
\usepackage{mdwlist}
\usepackage{tkz-graph}

\DeclareMathAlphabet{\eufrak}{U}{}{}{} 
\SetMathAlphabet\eufrak{normal}{U}{euf}{m}{n}
\SetMathAlphabet\eufrak{bold}{U}{euf}{b}{n}

\newcommand{\R}{\mathbb{R}}

\newcommand{\C}{\mathcal{C}}
\newcommand{\E}{\mathbb{E}}

\newcommand{\inte}{\mathbb{N}}

\makeatletter
\newcommand*\rel@kern[1]{\kern#1\dimexpr\macc@kerna}
\newcommand*\widebar[1]{
  \begingroup
  \def\mathaccent##1##2{
    \rel@kern{0.8}
    \overline{\rel@kern{-0.8}\macc@nucleus\rel@kern{0.2}}
    \rel@kern{-0.2}
  }
  \macc@depth\@ne
  \let\math@bgroup\@empty \let\math@egroup\macc@set@skewchar
  \mathsurround\z@ \frozen@everymath{\mathgroup\macc@group\relax}
  \macc@set@skewchar\relax
  \let\mathaccentV\macc@nested@a
  \macc@nested@a\relax111{#1}
  \endgroup
}
\makeatother

\makeatletter
\DeclareRobustCommand\widecheck[1]{{\mathpalette\@widecheck{#1}}}
\def\@widecheck#1#2{
    \setbox\z@\hbox{\m@th$#1#2$}
    \setbox\tw@\hbox{\m@th$#1
       \widehat{
          \vrule\@width\z@\@height\ht\z@
          \vrule\@height\z@\@width\wd\z@}$}
    \dp\tw@-\ht\z@
    \@tempdima\ht\z@ \advance\@tempdima2\ht\tw@ \divide\@tempdima\thr@@
    \setbox\tw@\hbox{
       \raise\@tempdima\hbox{\scalebox{1}[-1]{\lower\@tempdima\box
\tw@}}}
    {\ooalign{\box\tw@ \cr \box\z@}}}
\makeatother

\newtheorem{prop}{Proposition}
\newtheorem{lemma}{Lemma}
\newtheorem{definition}{Definition}
\newtheorem{corollary}{Corollary}
\newtheorem{theorem}{Theorem}
\newtheorem{remark}{Remark}

\def\({\left(}
\def\){\right)}

\def\[{\left[}
\def\]{\right]}

\newcommand{\id}{\textbf{Id}}

\def\P{\mathbb{P}}

\makeatletter
\newcommand*\bigcdot{\mathpalette\bigcdot@{.5}}
\newcommand*\bigcdot@[2]{\mathbin{\vcenter{\hbox{\scalebox{#2}{$\m@th#1\bullet$}}}}} 
\makeatother

\usepackage{cprotect}

\usepackage{comment}

\newcommand{\pt}{\partial}

\newcommand{\M}{\mathcal M}

\newenvironment{Proof}{\removelastskip\par\medskip
\noindent{\em Proof.} \rm}{\penalty-20\null\hfill$\square$\par\medbreak}

\allowdisplaybreaks

\GraphInit[vstyle = Shade]
\usetikzlibrary[intersections,
positioning,
petri,
backgrounds,
fit,
decorations.pathmorphing,
arrows,
arrows.meta,
bending,
calc,
intersections,
through,
backgrounds,
shapes.geometric,
quotes,
matrix,
trees,
shapes.symbols,
graphs,
math,
patterns,
external,
scopes,
matrix,
lindenmayersystems,
shapes.callouts,
shapes.misc,
angles,
shapes.arrows,
shadings]

\begin{document}
\title{
\huge
 On the random generation of Butcher trees 
} 

\author{
 Qiao Huang\footnote{
  School of Mathematics,
  Southeast University,
  Nanjing 211189,
  P.R. China.
  \\ \href{mailto:qiao.huang@seu.edu.cn}{qiao.huang@seu.edu.cn}
 } 
  \qquad
 Nicolas Privault\footnote{
 School of Physical and Mathematical Sciences, 
 Nanyang Technological University, 
 21 Nanyang Link, Singapore 637371.
\\ \href{mailto:nprivault@ntu.edu.sg}{nprivault@ntu.edu.sg}
}
}

\maketitle

\vspace{-0.5cm}

\begin{abstract} 
 The main goal of this paper is to provide an 
 algorithm for the random sampling of Butcher trees
 and the probabilistic numerical solution
 of ordinary differential equations (ODEs).
 This approach complements and simplifies a recent
 approach to the probabilistic representation of ODE solutions, 
 by removing the need to generate random
 branching times. 
 The random sampling of trees
 is compared to the finite order truncation
 of Butcher series in numerical experiments. 
\end{abstract}


\baselineskip0.7cm

\section{Introduction} 
\noindent
Butcher series \cite{butcher1963}, \cite{butcherbk} are used
to represent the solutions of ordinary differential 
equations (ODEs) 
by combining  rooted tree enumeration with Taylor expansions, see e.g. Chapters~4-6 of \cite{deuflhard}, and \cite{ehairer}
 and references therein for applications to 
 geometric numerical integration.
 Given $f:\R^d \to\R^d$ a smooth function and $t_0<T$,
consider the $d$-dimensional autonomous ODE problem 
\begin{equation}
  \label{ODE}
  \begin{cases}
    \dot x(t) = f(x(t)),
    & t \in(t_0,T],
      \smallskip
      \\
    x(t_0) = x_0\in\R^d. &
  \end{cases}
\end{equation}
If the solution $x(t)$ is sufficiently smooth at $t=t_0$,
Taylor's expansion yields 
\begin{equation}
  \label{fjl133} 
  x(t) = x_0 + (t-t_0) f(x(t_0)) + \sum_{k=2}^\infty \frac{(t-t_0)^k}{k!} \frac{d^{k-1} }{d s^{k-1}}f(x(s))_{\mid s=t_0}, 
\end{equation} 
  for small time $t-t_0$.
  The series \eqref{fjl133} can be rewritten using the
 ``elementary differentials'' 
$$
f, \ \nabla f(f), \ \nabla^2 f(f,f), \ \nabla f(\nabla f (f)), \ldots, 
$$
 as the expansion 
\begin{align} 
  \label{Taylor}
  x(t) & = x_0 + (t-t_0) f(x_0) + \frac{(t-t_0)^2}{2} \big(
  \nabla f(f) \big) (x_0)
  \\
  \nonumber 
   & \quad + \frac{(t-t_0)^3}{6} \big( \nabla^2 f(f,f) + \nabla f(\nabla f (f)) \big) (x_0) + \cdots, 
\end{align}
 which is known to admit a graph-theoretical expression 
 as the Butcher series 
 \begin{equation}
   \label{fjkld9} 
 x(t) = \sum_{\tau} \frac{(t-t_0)^{|\tau|}}{\tau!\sigma(\tau)} F(\tau)(x_0),
 \end{equation}
 over rooted trees $\tau$, 
 where $F(\tau)$ represents elementary differentials,
 and $\sigma (\tau )$, $\tau !$ respectively represent the symmetry and
  factorial of the tree $\tau$,
  see Section~\ref{s2} for definitions.
 The series \eqref{fjkld9} can be used to estimate
 ODE solutions by expanding $x(t)$ into a sum over trees
 up to a finite order.
 However, the generation of high order trees is
 computationally expensive.

 \medskip  

  In this paper, we consider the numerical estimation of the series
  \eqref{fjkld9} using Monte Carlo generation of 
  random trees and branching processes, 
  which are classical probabilistic tools that 
  have been the object of extensive studies,
  see for example \cite{vatutin}, \cite{vatutin2}. 
  Although Monte Carlo estimators cannot compete
   with classical Runge-Kutta schemes,
   they represent an alternative to the truncation
   of series, and they allow for estimates
   whose precision improves when the number
   of iterations increases.
    This approach is also motivated by related constructions
 extending the use of the 
 Feynman-Kac formula to the numerical estimation of the solutions
 of fully nonlinear partial differential equations
 by stochastic branching mechanisms and stochastic cascades,
 see 
 \cite{skorohodbranching},
 \cite{inw},
 \cite{hpmckean}, 
 \cite{sznitman},
 \cite{dalang},
 \cite{waymire},
 \cite{labordere},
 \cite{waymire1},
 \cite{waymire2},
 \cite{penent2022fully}.

   \medskip

 Here, in comparison with
 \cite{penent2022fully},
 we present a direct and simpler approach to the random generation
of Butcher trees that does not require the use of random branching times.
From a simulation point of view,
   this amounts to estimating \eqref{fjkld9}
   as an expected value, by generating random trees $\mathcal T$ 
   having a conditional probability distribution of the form  
\begin{equation}
\nonumber 
   \P ( \mathcal T = \tau \mid | \mathcal T | = n ) 
   =  \frac{c_n}{\tau!\sigma(\tau)}
   \end{equation} 
   over rooted trees $\tau$ of size $|\tau | = n\geq 0$,
   and $(c_n)_{n\geq 0}$ is
   a sequence of positive numbers. 
More precisely, we construct a random tree $\mathcal T$
 whose size complies with a probability distribution $(p_n)_{n\geq 0}$ on $\inte$, such that the solution of the ODE \eqref{ODE} admits the following probabilistic representation: 
\begin{equation}
  \label{random-Butcher-0}
  x(t) =
  \E \left[
 \frac{(t-t_0)^{|\mathcal T|} F(\mathcal T)(x_0) }{(|\mathcal T|\vee 1)p_{|\mathcal T|}}
 \right], 
\end{equation}
 see Theorem~\ref{2jkldf}. 
 The Monte Carlo implementation of \eqref{random-Butcher-0}
 allows us to estimate $x(t)$ at a given $t>t_0$
 within a (possibly finite) time interval. 
 A major difference with the expansion
 \eqref{fjkld9} is that the Monte Carlo method
 proceeds iteratively by randomly sampling trees of arbitrary
 orders, therefore avoiding the evaluation of
 \eqref{fjkld9} at a fixed order. 
 In comparison to 
 the probabilistic representation of
 ODE solutions proposed in \cite{penent4},
 the present algorithm
 does not require the generation of sequences of random
 branching times. 
 
\medskip

 In Section~\ref{s5}
 we extend our approach to semilinear ODEs of the form 
\begin{equation*} 
\begin{cases}
 \dot x(t) = Ax(t) + f(x(t)), & t \in(t_0,T],
  \smallskip
  \\
    x(t_0) = x_0\in\R^d,  &
  \end{cases}
\end{equation*}
where $A$ is a linear operator on $\R^d$.
Here, tree sizes are generated with
the Poisson distribution of mean $t>0$, and
the impact of the operator $A$ is taken into
account via an independent continuous-time Markov chain with generator $A$. 
In this case, our approach extends the construction
 presented in \cite{dalang} for linear PDEs, 
by replacing the use of linear chains (or paths) with general random trees 
in the setting of nonlinear ODEs. 
This construction can also be seen as a randomization of 
exponential Butcher series, see \cite{ostermann}, \cite{LO13}.

 \medskip

 Numerical examples are presented in Section~\ref{sec6},
 using Mathematica, by comparing the Monte Carlo evaluation
 of \eqref{random-Butcher-0} to the truncation
\begin{equation}
\label{truncated} 
     x(t) = \sum_{\tau\in\mathbf T\atop |\tau | \leq n} \frac{(t-t_0)^{|\tau|}}{\tau!\sigma(\tau)} F(\tau)(x_0), \qquad
     t>t_0, 
\end{equation}
     of \eqref{Butcher}
     at different orders $n\geq 1$. 
 The Mathematica codes presented in this paper 
 are available at 

\centerline{\href{https://github.com/nprivaul/mc-odes/blob/main/mc-odes.nb}{https://github.com/nprivaul/mc-odes/blob/main/mc-odes.nb}}

\noindent
 We refer the reader to
     \cite{ketcheson2022computing} for a complete implementation of Butcher series
     computations in Julia. 
     
 \medskip

 We proceed as follows. In Section~\ref{s2}
we review the construction of Butcher trees for the representation
of ODE solutions. Section~\ref{s3} reviews and proves additional
statements needed for labelled trees. 
In Section~\ref{s4} we present the algorithm for the random
generation of Butcher trees by the random attachment of vertices. 
Section~\ref{s5} deals with semilinear ODEs using Poisson
distributed tree sizes and a continuous-time Markov chain.  
 Numerical examples are presented in Section~\ref{sec6},
 and multidimensional versions of the codes are
 listed in Section~\ref{appendix}. 

\section{Butcher trees}
\label{s2}
\noindent
 In this section we review the construction of Butcher trees
 for the series representation \eqref{Taylor} of the solution of
 the ODE \eqref{ODE}. 
A rooted tree $\tau = (V,E, \bigcdot)$ is a nonempty set $V$ of vertices and a set of edges $E$ between some of the pairs of vertices, with a specific vertex called the root and denoted by ``$\bigcdot$'', such that the graph $(V,E)$ is connected with no loops. 
We denote by ``$\emptyset$'' and ``$\bigcdot$'' the empty tree and the single node tree, respectively. 
The next  definition uses the $B^+$ operation, see \cite[pages~44-45]{But21},
namely if $\tau_1, \ldots, \tau_m$ are trees,
then $[\tau_1, \ldots, \tau_m]$ denotes the tree $\tau$ formed by introducing a new vertex, which becomes the root of $\tau$, and $m$ new edges from the root of $\tau$ to each of the roots of $\tau_i$, $i = 1, 2, \ldots , m$.
We also use the notation
$$
[\tau_1^{k_1}, \ldots, \tau_n^{k_n}] = 
[\underbrace{\tau_1, \ldots, \tau_1}_{k_1\mbox{ \footnotesize terms}}, \ldots,
  \underbrace{\tau_n, \ldots, \tau_n}_{k_n\mbox{ \footnotesize terms}}],
\quad
k_1,\ldots , k_n \in \inte.
$$ 
\begin{definition}\cite[Definition~III.1.1]{ehairer}. 
 The set of rooted trees is denoted by $\mathbf T$, and can be defined
 as the closure of $\emptyset$ and $\bigcdot$ under the
 $B^+$ operation, i.e.: 
\begin{enumerate}[(i)]
\item $\emptyset \in\mathbf T$, $\bigcdot\in\mathbf T$,
\item
  $[\tau_1, \ldots, \tau_m]\in\mathbf T$ if $\tau_1, \ldots, \tau_m \in \mathbf{T}$. 
\end{enumerate}
\end{definition}
\noindent
The size (or order) of $\tau \in \mathbf{T}$ is
defined as the number of its vertices, and denoted by $|\tau |$.
In particular, we have $|\emptyset | = 0$ and
$|\bigcdot | = 1$. 
 For $n\geq 0$, we denote by $\mathbf T_n$ the subset of trees of order $n$
 in $\mathbf T$, and for $a_1,\ldots , a_m \in \R^d$ we let
$$ \nabla^m f(a_1 , \ldots , a_m )
 := \left(
\sum_{i_1 , \ldots , i_m=1}^d
\frac{\partial^m f}{\partial x_{i_1}
  \cdots \partial x_{i_m}}
a_{1,i_1}\cdots a_{m,i_m}
\right)_{j=1,\ldots ,d}.
$$
\begin{definition}\cite[Definition~III.1.2]{ehairer}.
\label{d1}
   The elementary differential of $f \in \C^\infty(\R^d, \R^d)$ is the mapping $F:\mathbf T \to
   \C^\infty(\R^d, \R^d)$ defined recursively by
\begin{enumerate}[(i)]
\item
  $F(\emptyset) = \id$, $F(\bigcdot) = f$,
\item
  $F(\tau) = \nabla^m f(F(\tau_1), \ldots, F(\tau_m))$ for $\tau = [\tau_1, \ldots, \tau_m]$.
\end{enumerate}
\end{definition}
\noindent
 The pair $(\tau,F(\tau))$ is called a Butcher tree,
 and the map $F$ can be used to express any of the terms involving $f$
 in the series \eqref{Taylor}. 
 Indeed, when $|\tau|=n$, $F(\tau)$ takes the form 
\begin{equation}\label{elem-diff}
  \prod_{i=1}^n \nabla^{m_i} f = \nabla^{m_1} f(\nabla^{m_2} f(\cdots), \ldots, \ldots(\ldots, f)\cdots),
\end{equation}
 for some integer sequence $(m_i)_{i=1,\ldots,n}$ such that
 $m_n=0$ and $m_1+\cdots + m_n = n-1$.

\medskip

 Given a tree $\tau\in\mathbf T$, the map $F$ provides a way to encode each vertex of $\tau$ using $f$ or its derivatives: each vertex with no descendants is coded by $f$, and the vertices with $m$ descendants are coded by $\nabla^m f$, for $m\geq 1$. 
 In order to characterize the coefficients of
 \eqref{Taylor}, we need two functionals defined on trees.
 \begin{definition}
\label{fjlkf4}
      \begin{enumerate}[a)]
      \item
        \cite[Section~304]{butcherbk}, 
        \cite[Section~2.5]{But21}. 
     The symmetry $\sigma$ of a tree is defined recursively by
\begin{enumerate}[(i)]
\item
  $\sigma(\emptyset) = 1$, $\sigma(\bigcdot) = 1$,
\item
  $\sigma([\tau_1^{k_1}, \ldots, \tau_m^{k_m}])
  = \prod_{i=1}^m k_i! \sigma(\tau_i)^{k_i}$ for $\tau_1, \ldots, \tau_m$ distinct
  and $k_1,\ldots , k_m \in \inte$.
\end{enumerate}
\item
   The factorial (or density) $\tau !$ of a tree $\tau$ is defined by
\begin{enumerate}[(i)]
\item
  $\emptyset !=1$, $\bigcdot\, !=1$,
\item
  $\tau! = |\tau| \prod_{i=1}^m \tau_i !$ for $\tau = [\tau_1, \ldots, \tau_m]$.
\end{enumerate}
\end{enumerate}
   \end{definition}
\noindent
 Using the above formalism, we obtain the following result. 
\begin{prop}
\label{fjlf1}
 \cite[Definition 3.4B, Theorem 3.5C]{But21}.  
 The series \eqref{Taylor} can be rewritten as 
 the Butcher series 
\begin{equation}
     \label{Butcher}
     x(t) = \sum_{\tau\in\mathbf T} \frac{(t-t_0)^{|\tau|}}{\tau!\sigma(\tau)} F(\tau)(x_0), \qquad
     t>t_0. 
\end{equation}
\end{prop}
\noindent
 The computation of the Taylor expansion \eqref{fjl133} can be
 implemented in the following Mathematica code,
 by noting that in order to calculate the term of
   order $k\geq 2$ in the
   Taylor expansion \eqref{fjl133}, the quantity $x(s)$
  in $f(x(s))$ can be replaced with its expansion until the order $k-1$,
  as the $(k-1)$-$th$ derivative $d^{k-1} / d s^{k-1}$
  of $(s-t_0)^l$ vanishes at $s=t_0$ when $l \geq k$. 

\medskip
 
\begin{lstlisting}[language=Mathematica]
  Taylor[f_, t_, x0__, t0_, k_] := (d = Length[x0]; 
  g[x_] := x0 + f[x0]*(x - t0); h = Array[hh, d]; 
  For[j = 2, j <= k, j++, For[i = 1, i <= d, i++, 
    h[[i]][z_] = g[x][[i]] + (z - t0)^j/j!*D[f[g[s]][[i]], {s, j - 1}] /. {s -> t0};]; v = {}; 
   For[i = 1, i <= d, i++, v = Append[v, h[[i]][x]]]; g[x_] = v]; 
  Return[g[t]])
f[x__] := {f1[x[[1]], x[[2]]], f2[x[[1]], x[[2]]]}
Taylor[f, t, {x1, x2}, 0, 2]
\end{lstlisting}
 with sample output 
\begin{align*}
  &
  \left\{
    \text{x1}+t \text{f1}(\text{x1},\text{x2})+
 \frac{1}{2} t^2 \left(\text{f1}^{(0,1)}(\text{x1},\text{x2}) \text{f2}(\text{x1},\text{x2})+\text{f1}(\text{x1},\text{x2})
  \text{f1}^{(1,0)}(\text{x1},\text{x2})\right),
  \right.
  \\
  &
  \left. \ \
  \text{x2}+t \text{f2}(\text{x1},\text{x2})+
  \frac{1}{2} t^2 \left(\text{f1}(\text{x1},\text{x2})
  \text{f2}^{(1,0)}(\text{x1},\text{x2})+\text{f2}(\text{x1},\text{x2}) \text{f2}^{(0,1)}(\text{x1},\text{x2})\right)
  \right\} 
\end{align*}
for $k=2$, $d=2$, and $t_0=0$.
Instead of the above code, we will use
the following implementation of the truncated 
 Butcher series
 \eqref{truncated} up to any tree order $n\geq 1$ in case $d=1$,
 which also prints out the corresponding trees. 

\medskip

\begin{lstlisting}[language=Mathematica]
  B[f_, t_, x0_, t0_, n_] := (If[n == 0, Return[x0], 
    If[n == 1, Return[x0 + (t - t0)*f[x0]], 
     sample = x0 + (t - t0)*f[x0]; g = Graph[{1 -> 2}]; 
     g = Graph[g, VertexLabels -> {1 -> D[f[ y], y]}]; 
     g = Graph[g, VertexLabels -> {2 -> f[y]}]; m = 1; 
     sample = sample + 1/2*(t - t0)^VertexCount[g]*
         Product[ff[[2]] , {ff, List @@@ PropertyValue[g, VertexLabels]}] /. {y -> x0}; 
     list = {g}; 
     While[m <= (n - 2), temp = list; list = {}; 
      Do[l = VertexCount[g]; 
       For[j = 1, j <= l, j++, gg = VertexAdd[g, {l + 1}]; 
        gg = Graph[gg, VertexLabels -> {l + 1 -> f[ y]}]; 
        lab = Sort[List @@@ PropertyValue[gg, VertexLabels]][[j]][[2]];
        gg = Graph[gg, VertexLabels -> {j -> D[lab, y]}]; 
        gg = EdgeAdd[gg, j -> l + 1]; Print[gg]; 
        sample = sample + (t - t0)^(l + 1)/(l + 1)!*
            Product[ff[[2]] , {ff, List @@@ PropertyValue[gg, VertexLabels]}] /. {y -> x0};
         list = Append[list, gg]], {g, temp}]; m = m + 1]; 
     Return[sample]]]);
\end{lstlisting}

\noindent 
 For example, the command \verb|B[f,t,x0,t0,4]| 
 produces the scalar output 

\vspace{-0.8cm} 

\begin{align*}
& \text{x0} +\text{f}(\text{x0})
   (\text{t}-\text{t0})+\frac{1}{2} \text{f}(\text{x0})
  (\text{t}-\text{t0})^2 \text{f}'(\text{x0})
  \\
  & + \frac{1}{6} \text{f}(\text{x0})^2
  (\text{t}-\text{t0})^3 \text{f}''(\text{x0})
  +\frac{1}{6} \text{f}(\text{x0})
  (\text{t}-\text{t0})^3 \text{f}'(\text{x0})^2
   \\
   &
    +\frac{1}{24}
   \text{f}(\text{x0}) (\text{t}-\text{t0})^4
   \text{f}'(\text{x0})^3
   +\frac{1}{24} \text{f}(\text{x0})^3 \text{f}^{(3)}(\text{x0})
   (\text{t}-\text{t0})^4+\frac{1}{6}
   \text{f}(\text{x0})^2 (\text{t}-\text{t0})^4 \text{f}'(\text{x0})
   \text{f}''(\text{x0})
   \end{align*} 

\noindent
and enumerates the trees appearing in the series \eqref{Butcher}. 

\smallskip

\begin{figure}[H]
\centering
\scalebox{0.7}{
  \begin{subfigure}[]{0.25\textwidth}
\includegraphics[width=\textwidth]{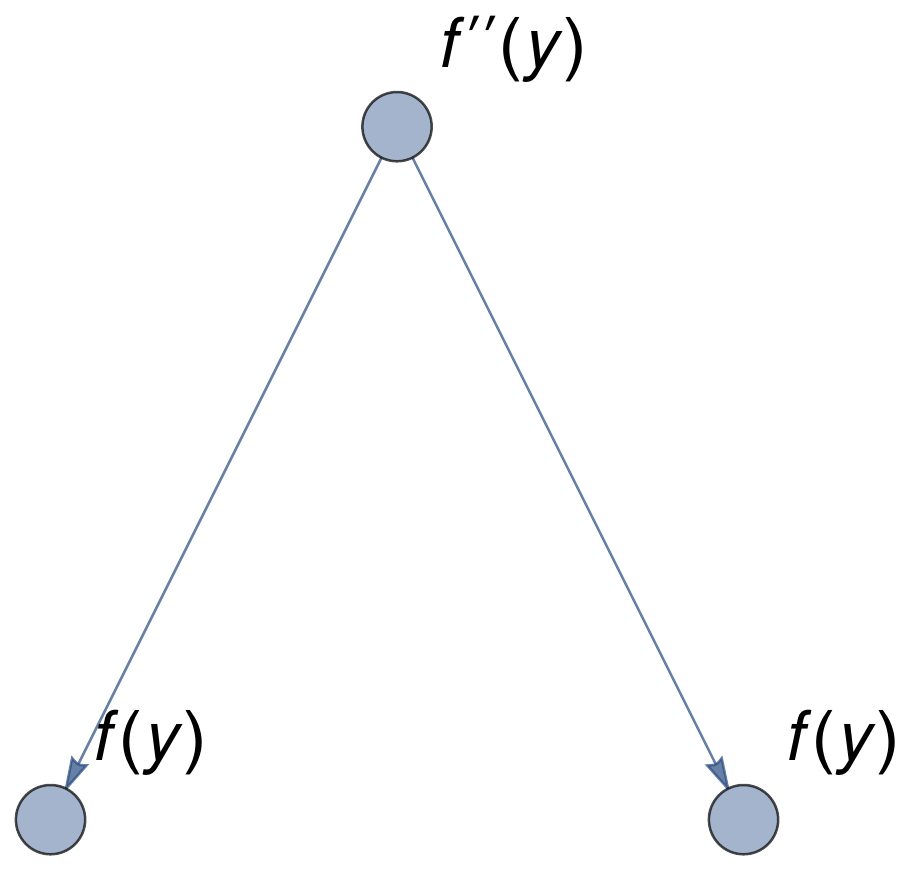} 
 \end{subfigure}
  \begin{subfigure}[]{0.27\textwidth}
\includegraphics[width=\textwidth]{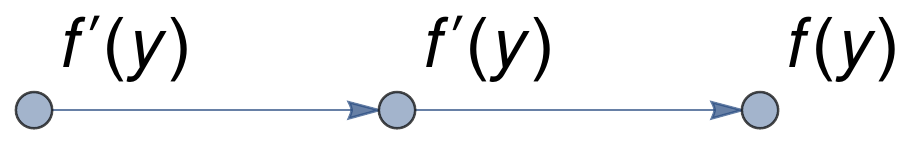} 
 \end{subfigure}
 \begin{subfigure}[]{0.30\textwidth}
\includegraphics[width=\textwidth]{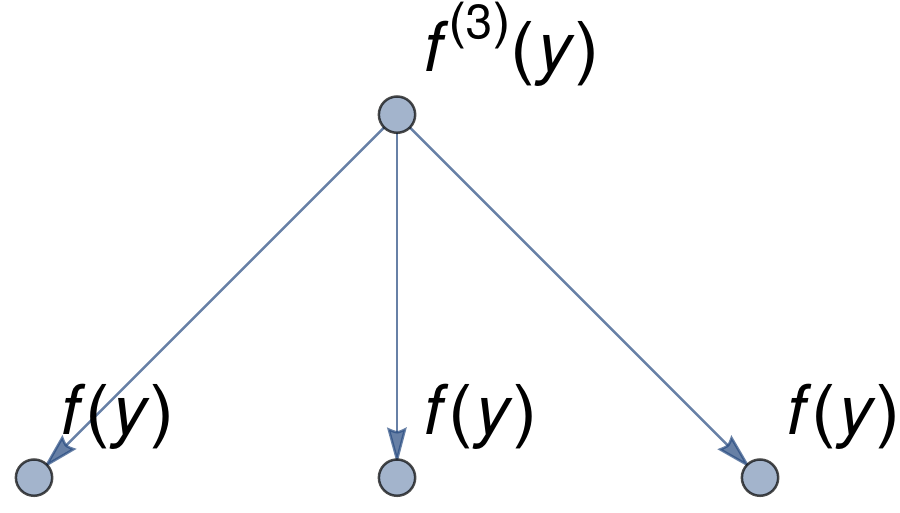} 
 \end{subfigure}
  \begin{subfigure}[]{0.30\textwidth}
\includegraphics[width=\textwidth]{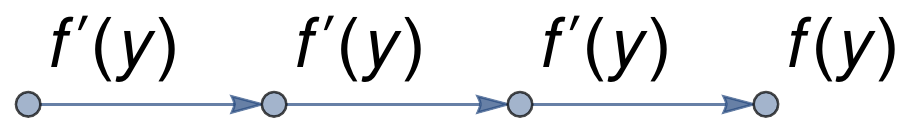} 
 \end{subfigure}
}
\scalebox{0.97}{
\begin{subfigure}[]{0.193\textwidth}
\includegraphics[width=\textwidth]{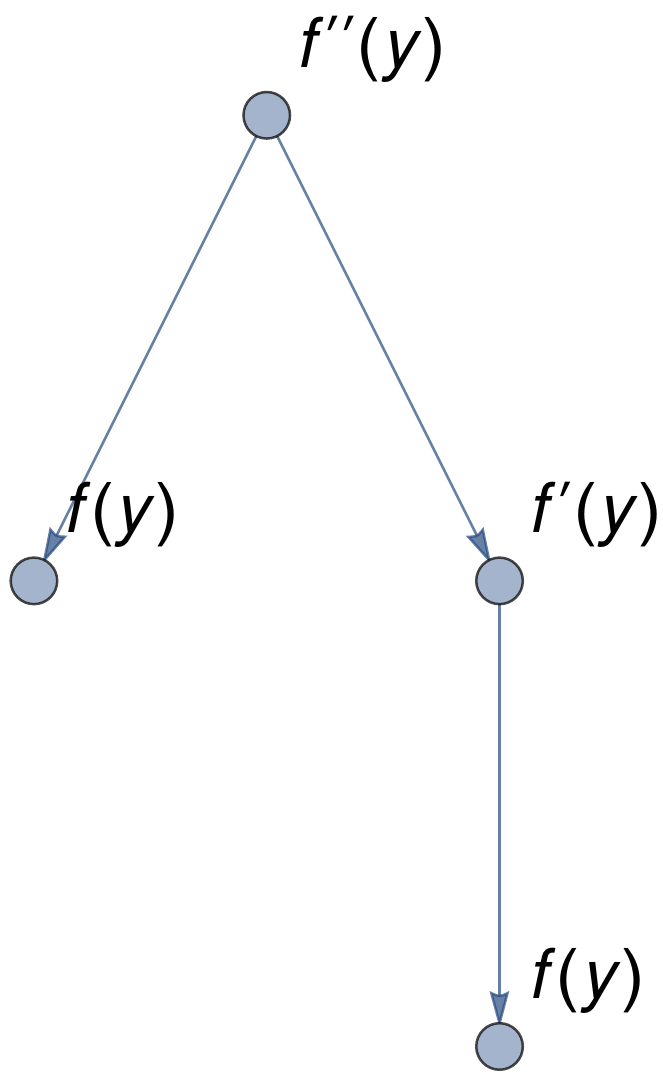} 
 \end{subfigure}
  \begin{subfigure}[]{0.19\textwidth}
\includegraphics[width=\textwidth]{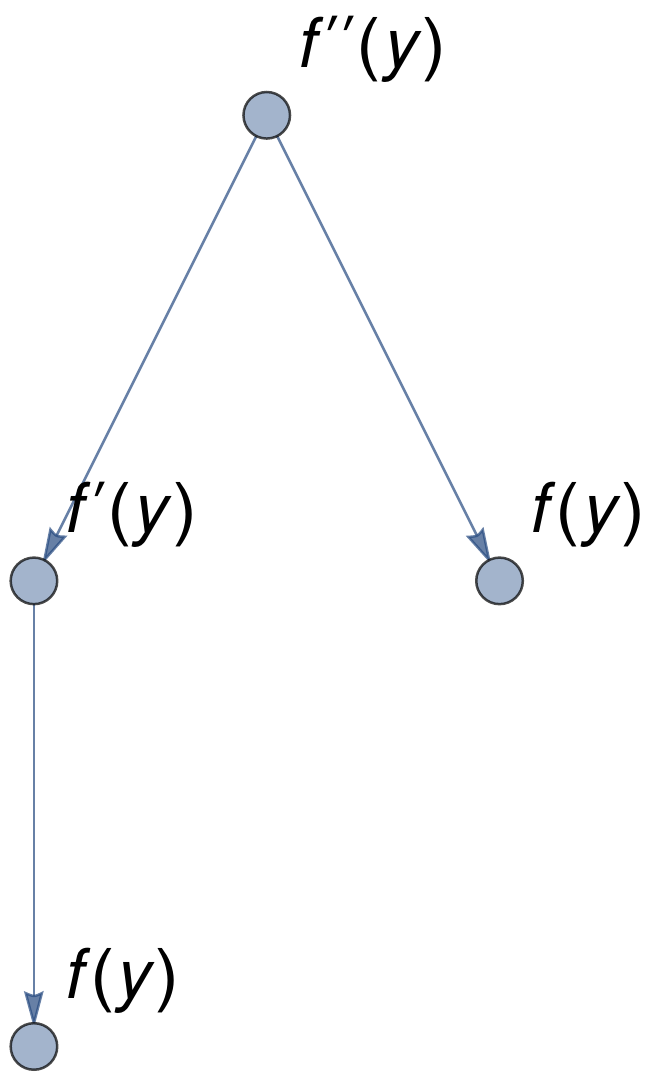} 
 \end{subfigure}
 \begin{subfigure}[]{0.19\textwidth}
\includegraphics[width=\textwidth]{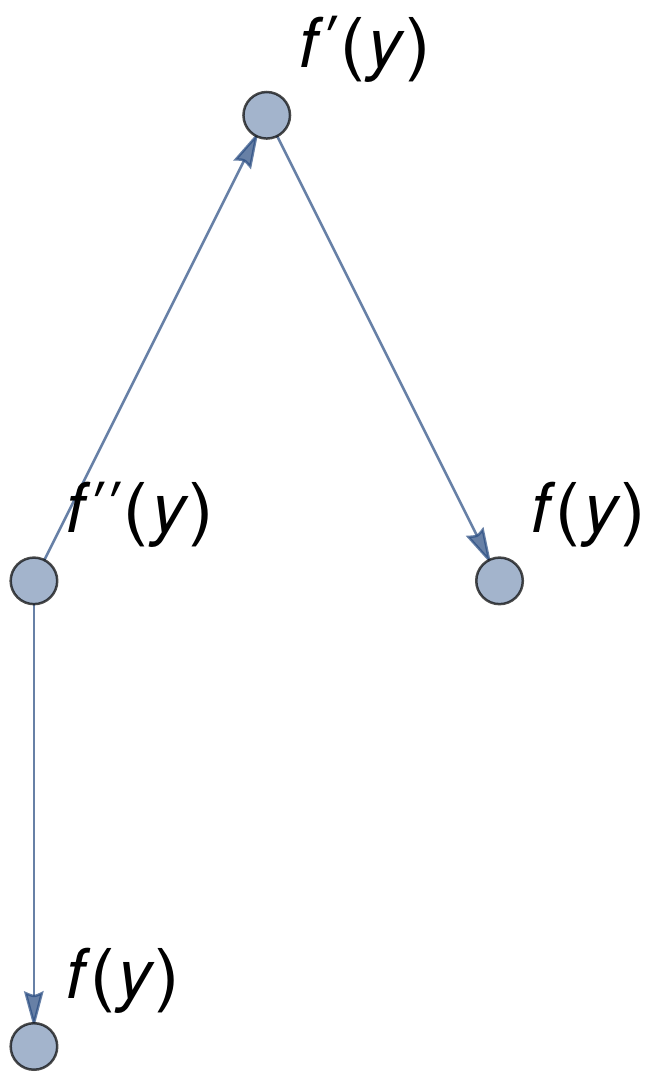} 
 \end{subfigure}
 \begin{subfigure}[]{0.282\textwidth}
\includegraphics[width=\textwidth]{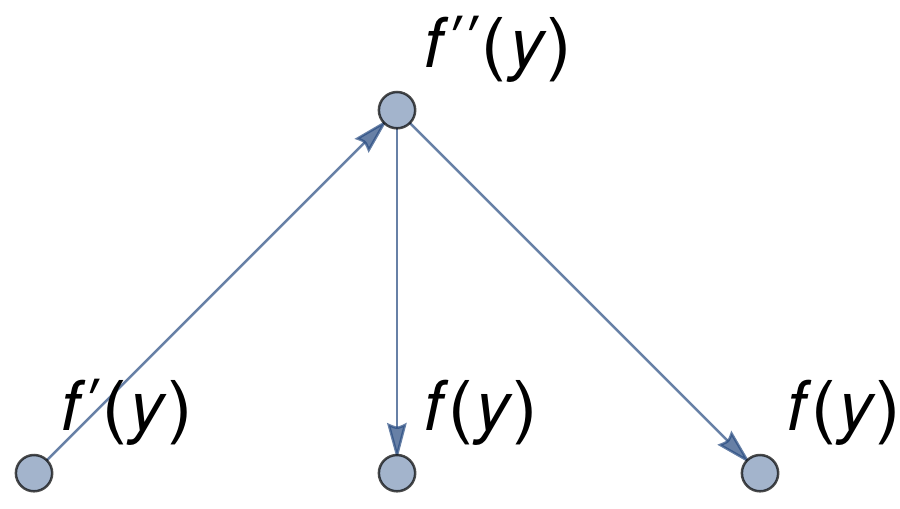} 
 \end{subfigure}
}
\vskip0.1cm
\caption{Generation of Butcher trees.}
\label{figexpl-2}
\end{figure}

\vskip-0.4cm

\noindent
Figure~\ref{figexpl-2} prints directed rooted trees
of orders $3$ and $4$, each with a natural orientation away from the root. 
In addition to the result of Proposition~\ref{fjlf1}, the derivatives of the Butcher series \eqref{Butcher} can be written using the elementary differentials $F$ introduced in Definition~\ref{d1}, as in the
following lemma. 
\begin{lemma}
  \cite[Lemma 3.5B]{But21} 
  We have
  \begin{equation}\label{der-x}
  x^{(n)}(t) = \sum_{\tau\in\mathbf T_n} \alpha ( \tau ) F(\tau)(x(t)), \quad n\geq 1, 
\end{equation}
  where
  \begin{equation}
    \label{fjkld34} 
  \alpha (\tau ) := \frac{|\tau |!}{\tau!\sigma(\tau)},
  \qquad
  \tau\in\mathbf T. 
\end{equation}
\end{lemma}
\begin{proof} 
  For completeness, we give a proof of \eqref{der-x}
  by induction using the Fa\`a di Bruno formula \cite[Theorem 2.1]{constantine},
  which states that for a smooth function $g: \R^d\to\R^{d'}$, the derivatives of $y(t)= g(x(t))$ are given by, for each $n\geq 1$, 
  \begin{align}
    \label{Bruno}
    & y^{(n)}(t)
    \\
    \nonumber 
     & \hskip-0.13cm = \sum_{k_1,\ldots , k_n \geq 0
    \atop
  \sum_{i=1}^n ik_i = n
  } \frac{n!}{\prod_{i=1}^n k_i!(i!)^{k_i}} \nabla^{k_1+\cdots +k_n} g(x(t)) \Big( \underbrace{\dot x(t), \ldots, \dot x(t)}_{k_1\mbox{ \footnotesize terms}}, \ldots, \underbrace{x^{(n)}(t), \ldots, x^{(n)}(t)}_{k_n \mbox{ \footnotesize terms}} \Big),
\end{align}
where the summation is taken over all different nonnegative solutions
$(k_1,\ldots, k_n)$ of the linear Diophantine equation $\sum_{i=1}^n ik_i = n$. 
 Now we prove \eqref{der-x}. We start with $n=1$, for which the result holds. Assume that the result is true for lower orders than $n$. Then we apply Fa\`a di Bruno's formula \eqref{Bruno} to derive
\begin{align*} 
 &   x^{(n)}(t) = [f(x)]^{(n-1)}(t)
    \\
     & = \sum_{k_1,\ldots , k_{n-1} \geq 0
    \atop
  \sum_{i=1}^{n-1} ik_i = n-1
    }
    \frac{(n-1)!}{\prod_{i=1}^{n-1} k_i!(i!)^{k_i}}
    \\
    & \qquad
    \sum_{\tau_i} \prod_{i=1}^{n-1} \left( \frac{i!}{\tau_i!\sigma(\tau_i)} \right)^{k_i} \nabla^{k_1+\cdots +k_{n-1}} f \big( F(\tau_1)^{k_1}, \ldots, F(\tau_{n-1})^{k_{n-1}} \big)(x(t)) \\
    &= \sum_{k_1,\ldots , k_{n-1} \geq 0
    \atop
  \sum_{i=1}^{n-1} ik_i = n-1
    } n!\frac{
      \nabla^{k_1+\cdots +k_{n-1}} f \big( F(\tau_1)^{k_1}, \ldots, F(\tau_{n-1})^{k_{n-1}} \big)(x(t))
    }{\left( n\prod_{i=1}^{n-1} (\tau_i!)^{k_i}\right) \left(
      \prod_{i=1}^{n-1} k_i!\sigma(\tau_i)^{k_i}\right)
      ]} \\
    &= \sum_{\tau\in\mathbf T_n} \frac{n!}{\tau!\sigma(\tau)} F(\tau)(x(t)), 
  \end{align*}
where in the first line, the first summation is taken over all nonnegative integers $(k_i)_{i=1,\ldots,n-1}$ satisfying $\sum_{i=1}^{n-1} ik_i = n-1$, the second summation is taken over all trees
$(\tau_i)_{i=1,\ldots,n-1}$ with $|\tau_i|=i$; in the last line $\tau = [\tau_1^{k_1}, \ldots, \tau_{n-1}^{k_{n-1}}]$ so that $|\tau| = 1+\sum_{i=1}^{n-1} ik_i = n$.
\end{proof} 
\noindent
The next lemma presents a tree expansion for $g(x(t))$, 
with $g \in \C^\infty(\R^d, \R^{d'})$, $d,d' \geq 1$. 
For this purpose, we consider the elementary differential
$F_g:\mathbf T \to \C^\infty(\R^d, \R^{d'})$
of $g$ 
 defined by 
\begin{equation}
  \label{fg} 
F_g(\emptyset ) := \id, \ \ 
F_g(\bigcdot) := g, \ \ 
\mbox{and}
 \ \  F_g(\tau) := \nabla^m g(F(\tau_1), \ldots, F(\tau_m)),
 \ \ \tau = [\tau_1, \ldots, \tau_m], 
\end{equation} 
with $\nabla^0 g = g$ and $0!=1$.
\begin{lemma}
  Let $g \in \C^\infty(\R^d, \R^{d'})$.
  We have
\begin{equation}\label{g-series}
 g(x(t))
 = \sum_{\tau \in \mathbf T\setminus\{\emptyset\}} \frac{(t-t_0)^{|\tau|-1}  }{(|\tau|-1) !} \alpha(\tau) F_g(\tau) (x_0). 
\end{equation}
\end{lemma}
\begin{proof} 
  Letting $x(t) = (x_i(t))_{1\leq i \leq d}$, 
  $x_0 = (x_i(0))_{1\leq i \leq d}$, 
  and denoting by $(x(t)-x_0)^{\otimes m}$
  the matrix
  $\big((x_{i_1}(t)-x_{i_1}(0)), \ldots , (x_{i_m}(t)-x_{i_m}(0))\big)_{1\leq i_1,\ldots , i_m \leq d}$,
  by \eqref{Butcher} and Definition~\ref{fjlkf4} we have 
\begin{equation*}
  \begin{split}
    g(x(t)) &= \sum_{m=0}^\infty \frac{1}{m !} \nabla^m g(x_0) \left( (x(t)-x_0)^{\otimes m} \right) \\
    &= \sum_{m=0}^\infty \frac{1}{m !} \sum_{\tau_1,\ldots \tau_m\in \mathbf T\setminus\{\emptyset\}} \frac{(t-t_0)^{\left|\tau_1\right|+\ldots+\left|\tau_m\right|}}{
       \prod_{i=1}^m ( \tau_i ! \sigma(\tau_i) ) }  \nabla^m g \left(F\left(\tau_1\right), \ldots, F\left(\tau_m\right)\right) (x_0) \\
    &= \sum_{m=0}^\infty \sum_{\tau_1,\ldots \tau_m\in \mathbf T\setminus\{\emptyset\}} \frac{|\tau| \prod_j k_j!}{m !\tau ! \sigma(\tau)}
    (t-t_0)^{|\tau|-1} \nabla^m g \left(F\left(\tau_1\right), \ldots, F\left(\tau_m\right)\right) (x_0),
  \end{split}
\end{equation*}
where the indexes $k_1,k_2,\ldots $ count
equal trees among $\tau_1, \ldots, \tau_m$.
From the fact that there are $\binom{m}{k_1,k_2, \ldots}$ possibilities for writing the tree $\tau$ in the form $[\tau_1, \ldots, \tau_m]$,
it then follows that 
 \begin{equation*}
  \begin{split}
 g(x(t)) 
 &=\sum_{\tau \in \mathbf T\setminus\{\emptyset\}}
 |\tau|
 \frac{ (t-t_0)^{|\tau|-1}}{\tau ! \sigma(\tau)} \nabla^m g \left(F\left(\tau_1\right), \ldots, F\left(\tau_m\right)\right) (x_0) 
 \\  &=\sum_{\tau \in \mathbf T\setminus\{\emptyset\}} \frac{(t-t_0)^{|\tau|-1}}{(|\tau|-1) !}
 \alpha(\tau) \nabla^m g \left(F\left(\tau_1\right), \ldots, F\left(\tau_m\right)\right) (x_0),
  \end{split}
\end{equation*}
 where $\alpha (\tau )$ is defined in
 \eqref{fjkld34}. 
 Then, we get \eqref{g-series}. 
\end{proof}
\noindent
 In particular, when $g$ is the function $f$ in ODE \eqref{ODE}, then $F_f = F$ and
\begin{equation*}
  f(x(t)) =
  \sum_{\tau \in \mathbf T\setminus\{\emptyset\}} \frac{(t-t_0)^{|\tau|-1}  }{(|\tau|-1) !} \alpha ( \tau ) F(\tau)(x_0),
\end{equation*}
which in turn proves that the Butcher series \eqref{Butcher} solves \eqref{ODE}.
\section{Labelled trees}
\label{s3}
\noindent
Our random generation
of Butcher trees will use
labelled trees, which provide a combinatorial interpretation
 of the coefficients appearing in \eqref{Butcher}.
By convention, the empty tree $\emptyset$ is labelled by $0$ 
 and the root of any non-empty tree is labelled by $1$. 
\begin{definition}
  \label{labelling}
  \cite[(2.5e)]{But21}
  For $n\geq 1$
  we denote by $\mathbf T^\sharp_n$ 
  the set of labelled rooted trees 
  $\tau$ of order $n$  with
  vertex sequence $V =\{1,\ldots,n\}$,
  written as $\tau = (V,E, 1)$, 
  such that the label of every vertex is smaller than that
  of each of its children.
  We also let $\mathbf T^\sharp := \bigcup_{n\geq 1} \mathbf T^\sharp_n$,
  and define the canonical forgetful map
  $\iota: \mathbf T^\sharp \to \mathbf T$
  by ``forgetting'' the labeling information contained
  in a labeled tree.
\end{definition}
\noindent
We note that the labelling of a tree $\tau$ is not necessarily unique.
 By abuse of notation, we shall omit the map $\iota$ when there is no ambiguity,
  in which case we will simply use $\tau$ to denote $\iota(\tau)$ for $\tau \in \mathbf T^\sharp$. 
\begin{prop}
  \label{fjlk13}
  \cite[Theorem~2.5F]{But21}
  The number of all possible labellings of a rooted tree
  $\tau = (V,E, \bigcdot)$
  is given by the coefficient $\alpha (\tau )$ defined in
  \eqref{fjkld34}.
\end{prop}
\noindent
 As a consequence of Proposition~\ref{fjlk13},
 since the labelling of a tree does not affect its elementary differential, we can rewrite \eqref{Butcher} and \eqref{der-x}
 respectively as 
\begin{equation}
  \label{Butcher-2}
  x(t) = \sum_{\tau\in\mathbf T^\sharp} \frac{(t-t_0)^{|\tau|}}{|\tau|!} F(\tau)(x_0)
  \quad
  \mbox{and}
  \quad
  x^{(n)}(t) = \sum_{\tau\in\mathbf T^\sharp_n} F(\tau)(x(t)), \quad n\geq 1, 
\end{equation}
 see also \cite[Theorem~II.2.6]{HNW93}. 
 Next, we define a new product on labelled trees
 that generalizes the beta-product 
 \cite[Section 2.1]{But21}, cf. 
 \S~1.5 of \cite{chapoton}. 
\begin{definition}[Grafting product]
  \nonumber 
  Let $\tau_1 = (V_1,E_1,1)$ and $\tau_2= (V_2,E_2,1)$ be two labelled trees,
  and let $l\in V_1 = \{1,\ldots, |\tau_1|\}$.
  \begin{itemize}
    \item 
      The grafting product with label $l$ of $\tau_1$ and $\tau_2$, denoted by $\tau_1 *_l \tau_2$, is the tree of order $|\tau_1| + |\tau_2|$ formed by grafting (attaching) $\tau_2$ from its root to the vertex $l$ of $\tau_1$, so that the vertices of $\tau_2$ become descendants of the vertex $l$.
    \item
      The tree $\tau_1 *_l \tau_2$ is labelled by keeping the labels of $\tau_1$, and by adding $|\tau_1|$ to the labels of $\tau_2$.
  \end{itemize}
\end{definition}
\noindent
 For any labelled tree $\tau$, we let $\emptyset *_0 \tau = \tau *_l \emptyset = \tau$ for all $0\leq l\leq |\tau|$, and keep the labels of $\tau$.
\begin{remark}
  \begin{enumerate}[(i)]
  \item The beta-product
        is a grafting product
    with label $1$,
    as the second tree is always attached to the root of the first one.
 \item  The $B^+$ operation can also be expressed by grafting-products, by forgetting labelling. For example, we have $[\tau_1,\tau_2]= \bigcdot *_1 \tau_1 *_1 \tau_2 = \bigcdot *_1 \tau_2 *_1 \tau_1$.
  \end{enumerate}
\end{remark}
\noindent 
 We note that any labelling is equivalent to a sequence of grafting of dots.
 In the next lemma we let $\triangle_0 := \{0\}$, and 
\begin{equation*}
  \triangle_n := \{(l_1, \ldots, l_n) \ : \ 1 \leq l_i \leq i, \ i=1,\ldots,n \}, \quad n\geq 1. 
\end{equation*}
\begin{lemma}
  \label{lemma-label}
  \begin{enumerate}[(i)] 
  \item
    Given $\tau$ a labelled tree with $|\tau|\ge2$, there is a unique sequence
    $(l_i)_{1 \leq i \leq |\tau| -1}$ in $\triangle_{|\tau|-1}$ such that 
  \begin{equation}\label{eqn-2}
    \tau = \bigcdot *_{l_1} \bigcdot *_{l_2} \cdots *_{l_{|\tau|-1}} \bigcdot.
  \end{equation}
\item
  For any $n\geq 2$, the map 
  which sends
    $\tau\in \mathbf T^\sharp_n$
  to the sequence $(l_1, \ldots, l_{n-1})$ determined by \eqref{eqn-2} is a bijection
  from $\mathbf T^\sharp_n$ to $\triangle_{n-1}$. 
  \end{enumerate}
 \end{lemma}
\begin{proof}
  We prove $(i)$ by induction on $|\tau|$. The case $|\tau|=2$
  is verified since $\tau = \bigcdot *_1 \bigcdot$.
  Suppose that \eqref{eqn-2} holds for all trees $\tau$ such that $|\tau|=n$,
  and let $\tau$ be a labelled tree with $|\tau|=n+1$. Denote by $\tau_-$ the subtree obtained by removing the vertex with label $n+1$ from $\tau$. It is clear from Definition~\ref{labelling} that the parent of the vertex $n+1$ has label $l$ not bigger than $n$,
   hence $\tau = \tau_- *_l \bigcdot$ has the form \eqref{eqn-2}.
   The converse of $(i)$ holds, since for each $n\ge2$, the sequence $(l_1, \ldots, l_{n-1}) \in\triangle_{n-1}$ determines a unique tree
   $\bigcdot *_{l_1} \bigcdot *_{l_2} \cdots *_{l_{n-1}} \bigcdot$
   in $\mathbf T^\sharp_n$. 
 Assertion $(ii)$ follows from $(i)$.
\end{proof}
 \noindent
 The next result is
 a consequence of Lemma~\ref{lemma-label}-$(ii)$. 
 \begin{corollary}
 The number of labelled trees of order $n\geq 1$  is given by 
$$
  |\mathbf T^\sharp_n| = \sum_{\tau \in \mathbf T^\sharp_n} 1 = \sum_{\tau \in \mathbf T_n} \alpha(\tau) = (n-1)!.
  $$
\end{corollary}
\begin{Proof} 
  For completeness, we provide a proof
  that does not rely on Lemma~\ref{lemma-label}-$(ii)$.
  By \eqref{der-x} and \eqref{fjkld34} we have
  $$
  x^{(n)}(t_0) = \sum_{\tau\in\mathbf T_n} \alpha(\tau) F(\tau)(x_0).
  $$ 
 Letting $f(x) := e^x$, $x_0:=0$ and $t_0:=0$, 
  by \eqref{elem-diff} we have $F(\tau) = f^n$ for all $\tau\in \mathbf T_n$.
  Hence, the solution
  $x(t) = -\log (1-t)$
  of \eqref{ODE} satisfies $x^{(n)}(0 ) = \sum_{\tau\in\mathbf T_n} \alpha(\tau)$,
  and it remains to note that $x^{(n)}(t) = (n-1)! (1-t)^{-n}$, $n\geq 1$. 
\end{Proof} 
 
\section{Random sampling of Butcher trees}
\label{s4}
\noindent
In this section we discuss the representation of solutions
to \eqref{ODE} by the random generation of Butcher trees.   
\begin{definition}
  \label{fjkldf} 
 Given $(p_n)_{n\geq 0}$ a probability distribution on $\inte$
 such that $p_n>0$ for all $n\geq 0$, 
 we generate a random labelled tree $\mathcal T$
 by uniform attachment, as follows. 
 \begin{enumerate}[i)]
 \item
   Choose the order of $\mathcal T$ with the
   distribution $\P( |\mathcal T| = n ) = p_n$, $n\geq 0$; 
\item Start from a root $\bigcdot$ with the label $1$;
\item Starting from a tree $\tau$ with order $l$,
   $1\leq l \leq n-1$,
   attach a new vertex with label $l+1$ to an
  independently and uniformly chosen vertex of $\tau$,
  and repeat this step inductively until we reach the given order $n$. 
\end{enumerate} 
\end{definition}
\noindent 
 For $n\geq 0$, we let 
 \begin{equation}
\nonumber 
 q^\sharp_n(\tau):= \P\left( \mathcal T = \tau \mid |\mathcal T| = n \right), \quad \tau\in \mathbf T^\sharp_n, 
\end{equation}
denote the conditional distribution of $\mathcal T$ given its
size is $n\geq 0$. 
We note that
 $\iota ( \mathcal T)$ is $\mathbf T$-valued,
and its conditional distribution on $\mathbf T$ is given by 
\begin{equation} 
  \label{random-tree-1}
  q_n(\tau) :=
  \P\left(
  \iota ( \mathcal T)
  = \tau \mid |\mathcal T| = n \right)
=
\sum_{
     \tau' \in \mathbf{T}^\sharp_n
    \atop
    \iota ( \tau' ) = \tau
  }
  q^\sharp_n(\tau'), \quad \tau\in \mathbf{T}_n. 
\end{equation}
 By Lemma~\ref{lemma-label}-$(ii)$, the 
random labelled tree $\mathcal T$
generated in Definition~\ref{fjkldf}
 takes the form 
    \begin{equation*}
    \mathcal T = \bigcdot *_{\eta_1} \bigcdot *_{\eta_2} \cdots *_{\eta_{|\mathcal T|-1}} \bigcdot, 
    \end{equation*}
 where $(\eta_1,\ldots, \eta_{|\mathcal T|-1})$
 is a uniform random variable taking values in
 $\triangle_{|\mathcal T|-1}$. 
\begin{theorem} 
\label{2jkldf}
Assume that there exists $C>0$
    such that
    \begin{equation}
      \label{fjlk3} 
           |\nabla^m f(x_0)|\leq C,
    \quad \mbox{for all } m\geq 0. 
\end{equation} 
 Then, the solution of the ODE \eqref{ODE} admits the probabilistic
    expression  
    \begin{equation}
      \label{random-Butcher}
  x(t) = \E \left[
 \frac{(t-t_0)^{|\mathcal T|} F(\mathcal T)(x_0) }{(|\mathcal T|\vee 1)p_{|\mathcal T|}}
 \right],
  \qquad t\in [t_0,t_0 + 1/C).
    \end{equation}
  \end{theorem} 
  \begin{Proof} 
    It follows from Lemma \ref{lemma-label}-$(ii)$ that
 given $|\mathcal T| = n$, 
    the random tree $\mathcal T$ is uniformly distributed in $\mathbf T^\sharp_n$, i.e. we have 
\begin{equation*}
  q^\sharp_n(\tau) = \P\left( \mathcal T = \tau \mid |\mathcal T| = n \right)
  = \frac{1}{|\mathbf T^\sharp_n|} = \frac{1}{( (n-1)\vee 0 )!},
\end{equation*}
in which case $q^\sharp_n(\tau)$ is independent of $\tau\in\mathbf T^\sharp_n$,
and the conditional probability \eqref{random-tree-1} is given by
\begin{equation*}
  q_n(\tau) = \frac{\alpha(\tau)}{ ( (n-1)\vee 0 )!},
   \qquad \tau\in \mathbf T. 
\end{equation*}
 Hence, we have 
\begin{align}
 \nonumber 
\E \left[
 \frac{(t-t_0)^{|\mathcal T|} F(\mathcal T)(x_0) }{(|\mathcal T|\vee 1)p_{|\mathcal T|}}
 \right]
 &= \sum_{n=0}^\infty
\frac{(t-t_0)^n p_n}{(n\vee 1)p_n} 
    \sum_{\tau\in \mathbf T_n^\sharp }
    q_n^\sharp (\tau ) \E \left[ F(\mathcal T)(x_0) \mid |\mathcal T| = n, \mathcal T=\tau \right] 
  \\
  \nonumber
  & = 
  \sum_{\tau\in\mathbf T^\sharp} \frac{(t-t_0)^{|\tau|}}{|\tau|!} F(\tau)(x_0)
\\
 \label{prob-rep-Butcher-2}
  & = 
    x(t), 
\end{align} 
 by the first equation of \eqref{Butcher-2}. 
 From the assumption \eqref{fjlk3}, we have
 $|F(\tau)(x_0)| \leq C^{|\tau|}$ for all $\tau\in\mathbf T$
 such that $|\tau |\geq 1$. The
 $q$-$th$ integrability of \eqref{prob-rep-Butcher-2},
 $q\geq 1$, can be implied by the bound 
\begin{eqnarray} 
  \nonumber
    \E \left[
  \left|
  \frac{(t-t_0)^{|\mathcal T|} F(\mathcal T)(x_0) }{(|\mathcal T|\vee 1)p_{|\mathcal T|}}
  \right|^q
\right]
 & \le & 
\frac{|x_0|^q}{p_0^{q-1}}
+ \sum_{n=1}^\infty \frac{( C (t-t_0))^{nq} }{n ^qp_n^{q-1}} 
\sum_{\tau\in \mathbf T^\sharp_n} q^\sharp_n(\tau) 
\\
\label{jkfl3}
 & = & \frac{|x_0|^q}{p_0^{q-1}}
 + \sum_{n=1}^\infty \frac{(C(t-t_0))^{nq}}{n^q p_n^{q-1}}, 
\end{eqnarray} 
 which is finite for $q=1$, provided that $C(t-t_0)<1$.
\end{Proof} 
   \noindent
  The random generation of Butcher trees in
   Theorem~\ref{2jkldf} is implemented in
   the following Mathematica code: 

   \smallskip
    
\begin{lstlisting}[language=Mathematica]
  MCsample[f_, t_, x0_, dist_] := (n = RandomVariate[dist]; 
   If[n == 0, Return[x0/PDF[dist, 0]], 
    If[n == 1, Return[t*f[x0]/PDF[dist, 1]], g = Graph[{1 -> 2}]; 
     g = Graph[g, VertexLabels -> {1 -> D[f[ y], y]}]; 
     g = Graph[g, VertexLabels -> {2 -> f[y]}]; m = 1; 
     While[m <= (n - 2), l = VertexCount[g]; 
      j = RandomVariate[DiscreteUniformDistribution[{1, l}]]; 
      g = VertexAdd[g, {l + 1}]; 
      g = Graph[g, VertexLabels -> {l + 1 -> f[ y]}]; 
      lab = Sort[List @@@ PropertyValue[g, VertexLabels]][[j]][[2]];
      g = Graph[g, VertexLabels -> {j -> D[lab, y]}]; 
      g = EdgeAdd[g, j -> l + 1]; m++]; 
     sample = Product[ff[[2]] , {ff, 
         List @@@ PropertyValue[g, VertexLabels]}] /. {y -> x0}; 
      Return[sample*t^n/PDF[dist, n]/n]]]);
f[y_] := Exp[y]
MCsample[f, t, x0, GeometricDistribution[0.5]]
\end{lstlisting}

\section{Connection with semilinear PDEs} 
\label{s5}
\noindent
In this section, we consider the case where
the function $f$ in \eqref{ODE} involves a linear component,
i.e. $f(x) = Ax + g(x)$, where $A$ is a linear operator on $\R^d$,
in which case the ODE \eqref{ODE} becomes
\begin{equation}\label{sl-ODE}
\begin{cases}
 \dot x(t) = Ax(t) + g(x(t)), & t \in (t_0,T],
  \smallskip
  \\
    x(t_0) = x_0\in\R^d,  &
  \end{cases}
\end{equation}
 and can be rewritten in integral form as 
\begin{equation*}
  x(t) = e^{(t-t_0)A} x_0 + \int_{t_0}^t e^{(t-s)A} g(x(s)) ds,
  \quad
  t\in (t_0,T]. 
\end{equation*}
\noindent
 By \cite[Theorem 4.5]{LO13} we have 
\begin{equation}
\label{fjkl1} 
 x(t) = 
 \sum_{\tau \in \mathbf T } \alpha(\tau) \phi_{|\tau|}(t,A) F_g (\tau) (x_0),
\end{equation}
 where
 $F_g$ is defined in \eqref{fg},
 with $\phi_0(t,a) := e^{(t-t_0)a}$ and 
\begin{equation} 
  \label{phi}
  \phi_n(t,a) := 
\int_{t_0\leq t_1< \cdots <t_n\leq t} e^{(t-t_n)a} dt_n \cdots dt_1
= 
    \int_{t_0}^t e^{(t-s)a} \frac{(s-t_0)^{n-1}}{(n-1) !} ds , 
\end{equation} 
 $n\geq 1$, $t \geq t_0$.
 In addition, from the fact that
labelling does not change elementary differentials,
the expansion \eqref{fjkl1} can be rewritten as the exponential Butcher series 
 \begin{equation}\label{exp-B-series}
  x(t) 
  = \sum_{\tau\in\mathbf T^\sharp} \phi_{|\tau|}(t,A) F_g (\tau)(x_0).  
\end{equation}
\indent
 Given $(N_t)_{t\ge t_0}$ 
  a standard Poisson process with
 $$
 \P (N_t = n) = e^{-(t-t_0)} \frac{(t-t_0)^n}{n!}, 
 \quad
 t\geq t_0, \ n\geq 0,
 $$ 
and increasing sequence of jump times
$(T_i)_{i\geq 1}$, and let $T_0 = t_0$,
let ${\cal T}_t$ denote the random tree constructed
 in Definition~\ref{fjkldf},
 using the Poisson distribution $p_n = \P (N_t = n)$, $n\geq 0$.
 In what follows, we assume that $A$ is a stochastic matrix, that is, a square matrix with non-negative entries where each column sums up to $1$, which generates a continuous-time Markov chain $X = (X_t)_{t\ge t_0}$, independent of
   $(N_t)_{t\ge t_0}$. 

 \medskip

  In Theorem~\ref{djk13} we propose a canonical way to
 evaluate the solution to the semilinear equation \eqref{sl-ODE}
 as an expected value over random trees.
 It is worth noting that
   the decomposition $f(x) = Ax + g(x)$ 
   can be used for a generalization to
   semilinear parabolic PDEs,
   in which case $A$ is an elliptic operator that
   can generate a Markov process $X = (X_t)_{t\ge t_0}$,
   and the discrete $\{1,\ldots , d\}$-valued
   index $i$ is replaced by the spatial variable of the PDE. 
 This can also be regarded as a randomization of the 
 exponential Butcher series \eqref{fjkl1}, 
 and as a nonlinear extension of
 the probabilistic representation of \cite{dalang}
 which uses linear chains for linear PDEs. 
 In the special case $A=0$, this
 probabilistic representation recovers \eqref{random-Butcher} 
 by generating tree sizes via the Poisson distribution 
 $(p_n)_{n\geq 0}$ with parameter $t-t_0$.
 \begin{theorem}
  \label{djk13}
  Assume that $A$ is a stochastic matrix and there exists $C>0$
    such that
    \begin{equation}
      \label{fjlk3-0} 
       |\nabla^m g(x_0)| + |Ax_0| + |A| \leq C,
    \quad \mbox{for all } m\geq 0. 
\end{equation} 
  Then, for $t\in [t_0,t_0+1/C)$ we have
  \begin{equation}
    \label{fjklf} 
    x_i (t) = e^{(t-t_0)} \E\big[ ((|{\cal T}_t|-1)\vee 0 )!
      \left( F_g ({\cal T}_t)(x_0) \right)_{X_{t-T_{|{\cal T}_t|}}}
      \mathbf{1}_{\{T_{|\mathcal T_t|}\le t\}}
        \ \! \big| \ \!  X_{t_0} = i \big],
\end{equation} 
 $i =1,\ldots , d$.
\end{theorem}
\begin{Proof} 
 From the fact that the sequence $(T_i-T_{i-1})_{i=1, \ldots, n}$ is i.i.d. with
 common exponential distribution,
 for any integrable function $h$ on the $n$-dimensional simplex 
 $$
 \triangle^n_t := \{(t_1, \ldots , t_n): t_0 \leq t_1 < \cdots <t_n \leq t\},
 $$ 
 we have 
\begin{align*} 
 & 
  \E[ \mathbf{1}_{\{N_t = n \}}
    h(T_1,\ldots,T_n) ]
   \\
  & =  \E\left[
    \mathbf{1}_{\{
      t_0 < T_n \leq t < T_{n+1} 
      \}}
    h\left(T_1,\ldots,t_0 + \sum_{i=1}^n (T_i-T_{i-1}) \right)
     \right]
  \\
  & =  \int_{0< r_1 + \cdots + r_n \leq t-t_0} h\left(t_0 + r_1,\ldots,t_0 + \sum_{i=1}^n r_i \right) e^{-(r_1+\cdots+r_n)}
  \\
  & \qquad \qquad \qquad \qquad  \qquad
  \qquad \qquad \qquad \qquad  
  \int_{t-t_0-(r_1+\cdots+r_n)}^\infty e^{-r_{n+1}} dr_{n+1} dr_n \cdots dr_1
  \\
  & =  e^{-(t-t_0)} \int_{0< r_1 + \cdots + r_n \leq t-t_0} h\left(t_0 + r_1,\ldots,t_0 + \sum_{i=1}^n r_i \right) dr_n \cdots dr_1
  \\
    &=  e^{-(t-t_0)} \int_{t_0\leq t_1 < \cdots <t_n \leq t} h(t_1,\ldots,t_n) dt_n \cdots dt_1,
\end{align*}
where we applied the change of variables $t_i = t_0 + r_1+ \cdots + r_i$ in the last equality. Taking
$h(t_1,\ldots , t_n): = e^{(t-t_n)a}$, 
$t_0\leq t_1 < \cdots < t_n \leq t$, 
 it follows that \eqref{phi} can be rewritten as 
\begin{equation*}
  \phi_n(t,a) = e^{t-t_0} \E\left[
    \mathbf{1}_{\{
      N_t = n
      \}
      }
      e^{(t-T_n)a}  \right], \quad n\geq 0.
\end{equation*}
 Next, by construction of the continuous-time
 Markov chain $(X_t)_{t\ge t_0}$ with generator $A$, we have 
\begin{equation*}
  \big( e^{(t-t_0)A} x \big)_i = \E\left[ x_{X_t} \ \! \big| \ \! X_{t_0} = i \right], \quad i = 1,\ldots, d,
  \ 
  x =(x_1,\ldots , x_d) \in\R^d. 
\end{equation*}
Finally, as the random tree ${\cal T}_t$ is constructed with
 the Poisson random size $N_t$ and independent uniform attachment, we have 
\begin{equation*}
  |{\cal T}_t| = N_t, \qquad \P\left( {\cal T}_t = \tau \mid |{\cal T}_t| = n \right) = \frac{1}{|\mathbf T^\sharp_n|}, \qquad \tau\in\mathbf T^\sharp_n.
\end{equation*}
 Combining the above with \eqref{exp-B-series}, we get 
\begin{eqnarray*}
     x_i(t) & = & e^{t-t_0} \sum_{n=0}^\infty \sum_{\tau\in\mathbf T^\sharp_n} \E\big[
      \mathbf{1}_{ \{ N_t = n \} }
      \big( e^{(t-T_n)A} F_g (\tau)(x_0) \big)_i  \big]
  \\
  &= & e^{t-t_0} \sum_{n=0}^\infty \sum_{\tau\in\mathbf T^\sharp_n} 
  \E\big[
    \mathbf{1}_{ \{ N_t = n \} }
    \left( F_g (\tau)(x_0) \right)_{X_{t-T_n +t_0}} \ \! \big| \ \! X_{t_0} = i \big] \\
    &= & e^{t-t_0} \sum_{n=0}^\infty \sum_{\tau\in\mathbf T^\sharp_n} \E\big[ \left( F_g (\tau)(x_0) \right)_{X_{t-T_n +t_0}} \ \! \big| \ \!  N_t = n, X_{t_0} = i \big] \P(N_t = n) \\
  &= &
  e^{t-t_0}
  \sum_{n=0}^\infty
  ((n-1)\vee 0)! 
  \P(|{\cal T}_t| = n)
  \\
  & &
  \qquad
   \times 
   \sum_{\tau\in\mathbf T^\sharp_n}\E\big[ \left( F_g (\tau)(x_0) \right)_{X_{t-T_n
          +t_0}} \ \! \big| \ |{\cal T}_t| = n,
     {\cal T}_t = \tau, X_{t_0} = i \big]
  \P\left( {\cal T}_t = \tau \mid | {\cal T}_t| = n \right) 
  \\ 
  &= & e^{t-t_0} \E\big[ ((|{\cal T}_t|-1)\vee 0 )! \left( F_g ({\cal T}_t)(x_0)
    \right)_{X_{ t - T_{|\mathcal T_t|} + t_0}}
    \mathbf{1}_{\{T_{|\mathcal T_t|}\le t\}}
   \ \! \big| \ \!  X_{t_0} = i \big]. 
\end{eqnarray*} 
 By the definition \eqref{fg} of $F_g$ and
the bound \eqref{fjlk3-0},
the $q$-$th$ integrability of \eqref{fjklf}, 
 $q\geq 1$, can be controlled by the bound 
\begin{eqnarray} 
  \nonumber
  \lefteqn{
    \! \! \! \! \! \! \! \! \! \! \! \! \! \! \! \! \! \! \! \! \! \! \! \! \! \!
    \E\big[ \big|
      ((|{\cal T}_t|-1)\vee 0 )! \left( F_g ({\cal T}_t)(x_0)
      \right)_{X_{t - T_{|\mathcal T_t|} + t_0}}
      \big|^q \ \! \big| \ \!  X_{t_0} = i \big]
  }
  \\
  \nonumber
   & \le & 
   e^{-(t-t_0)} |x_0|^q
  + \E\big[ \mathbf{1}_{ \{ |{\cal T}_t| \geq 1 \} }
    \big|
    ((|{\cal T}_t|-1)!)^q
    C^{|{\cal T}_t|}
    \big|^q \big]
  \\
  \nonumber
  & = &    e^{-(t-t_0)} |x_0|^q + 
  e^{-(t-t_0)}
  \sum_{n=1}^\infty (n-1)!^q\frac{C^{nq}}{n!}
  (t-t_0)^n, 
\end{eqnarray} 
 which is finite for $q=1$, provided that $C(t-t_0)<1$.
\end{Proof}

\section{Numerical examples}
\label{sec6}
\noindent
In this section we consider numerical
implementations of the Monte Carlo generation of
Butcher trees for problems of the form \eqref{ODE}. 
\begin{enumerate}[i)]
  \item Let $f(y) := e^y$, and consider the equation 
   \begin{equation}
     \label{e1} 
\dot{x}(t) = e^{x(t)}, \quad x(0) = x_0, \quad t_0=0, 
\end{equation} 
 with solution
$$
 x(t) = -\log ( e^{-x_0} - t),
 \qquad
 t\in [0,e^{-x_0}).
   $$
\noindent
 In this case, the moment bound \eqref{jkfl3} is sharp 
 with $C = e^{x_0}$. 
\item
  Let $f(t,y) := yt + y^2$,
 and consider the equation
  \begin{equation}
    \label{e2} 
\dot{x}(t) = t x(t) + x^2(t),  \qquad x(0) = 1/2, \quad t_0=0, 
\end{equation} 
 with solution 
$$
x(t) = \frac{e^{t^2/2}}{2-\int_0^t e^{s^2/2} ds}, 
$$
 see Eq.~(223a) in \cite{butcherbk}. 
\end{enumerate}

\noindent
Table~\ref{table:example 2} displays the growth of
computation times for the command \verb|B[f,t,x0,t0,n]|
applied to \eqref{e1} with $x_0=1$, 
and to \eqref{e2} with $x_0=1/2$, $n=1, \ldots , 8$.  
For the purpose of benchmarking,
all tree generations are performed using Mathematica.
 
\begin{table}[H]
    \centering
    \resizebox{\textwidth}{!}{
      {\begin{tabular}{|c|c|c|c|c|c|c|c|c|c|}
        \hline
        $n$ & 1 & 2 & 3 & 4 & 5 & 6 & 7 & 8 & MC (Geometric) 
        \\
        \hline
        ~Eq. \eqref{e1}, $d=1$~ & ~0s~ & ~0s~ & ~0.1s~ & ~0.1s~ & ~0.4~ & ~0.5s~ & ~3s~ & ~21s~ & ~22s (70K samples)~ 
        \\
        \hline
        Eq. \eqref{e2}, $d=2$ & 0s & 0s & 0s & 0.2s & 1s & 13s & ~222s~ & ~$> 1$h~ & ~164s (10K samples)~ 
        \\
                \hline
      \end{tabular}}
      }
	\caption{Computation times in seconds for \eqref{Butcher} applied
          to \eqref{e1} and \eqref{e2}. 
        }
\label{table:example 2}
\end{table}

\vskip-0.3cm

\noindent
 Figure~\ref{fig5} 
 compares the numerical solutions
 of \eqref{e1} and \eqref{e2}
 by the truncated Butcher series expansion 
$$
 x(t) = \sum_{\tau\in\mathbf T\atop |\tau | \leq n} \frac{(t-t_0)^{|\tau|}}{\tau!\sigma(\tau)} F(\tau)(x_0), \qquad
     t>t_0, 
$$
     denoted by B-$n$, 
     to the probabilistic representation \eqref{random-Butcher},
     for different orders $n\geq 1$. 
 The Monte Carlo estimations of  
 \eqref{random-Butcher} use the geometric distribution with
 respectively $70,000$ and
 $10,000$ samples, see Table~\ref{table:example 2}, so that
 their runtimes are 
 comparable to those of the Butcher series estimates.  
 The solution of \eqref{e1} is estimated using the above codes
 for one-dimensional ODEs,
 and the solution of \eqref{e2} is estimated using
 the multidimensional codes presented in Section~\ref{appendix}, 
 after rewriting the non-autonomous ODE \eqref{e2}
 as a two-dimensional autonomous system. 
 
\begin{figure}[H]
\centering
\begin{subfigure}{.49\textwidth}
\centering
\includegraphics[width=\textwidth]{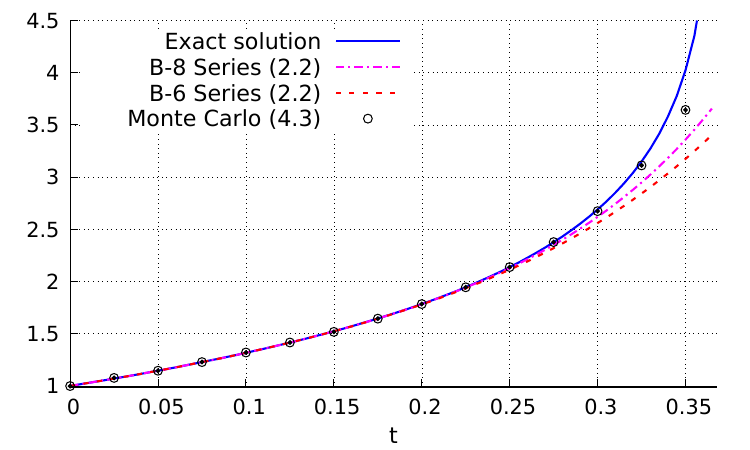}
\vskip-0.1cm
\cprotect\caption{Numerical solutions of
    \eqref{e1} with $x_0=1$.}
\end{subfigure}
\hskip-0.2cm
\begin{subfigure}{.49\textwidth}
\centering
\includegraphics[width=\textwidth]{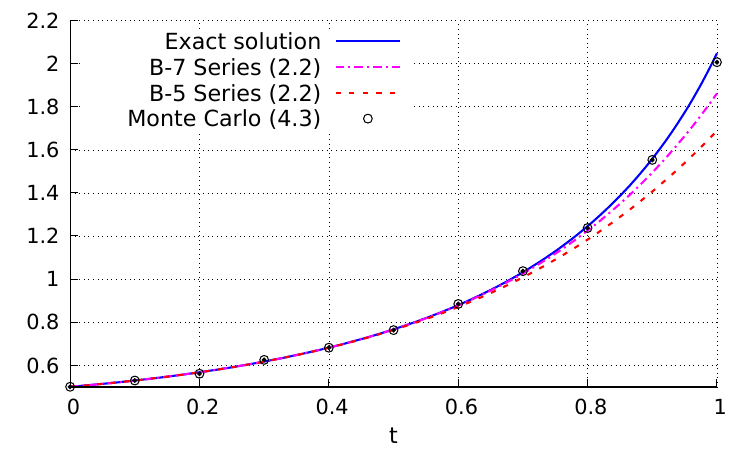}
\vskip-0.1cm
\cprotect\caption{Numerical solutions
    of \eqref{e2} with $x_0=1/2$.}
\end{subfigure}
\caption{Comparisons of
 \eqref{Butcher} {\em vs.} \eqref{random-Butcher}.} 
\label{fig5}
\end{figure}

\vskip-0.3cm

\noindent
 Next, we compare the performance of various probability distributions
$(p_n)_{n\in \inte}$ in terms of variance. 
\subsubsection*{Variance analysis} 
\noindent
$(i)$ {\em Poisson distribution}.
Taking $p_n := \lambda^n e^{-\lambda} / n!$, $n\geq 0$, 
 to be the Poisson distribution with parameter $\lambda > 0$, 
  the variance bound \eqref{jkfl3} is given by the series
$$ 
  \frac{x_0^2}{p_0} +
  \sum_{n=1}^\infty \frac{(C(t-t_0))^{2n}}{n^2 p_n}
  =
  \frac{x_0^2}{p}
  + e^{\lambda} \sum_{n=1}^\infty \left(
  \frac{C^2(t-t_0)}{\lambda} \right)^n \frac{(n-1)!}{n}, 
  $$
  which diverges for all $t>t_0$.
\\
\noindent $(ii)$
{\em Geometric distribution}. 
  Taking $p_n := (1-p)p^n$, $n\geq 0$,
  to be the geometric distribution with success probability $1-p$
  for some $p\in[0,1)$, 
  the variance bound \eqref{jkfl3} is given by the series
$$ 
  \frac{x_0^2}{p_0} +
  \sum_{n=1}^\infty \frac{(C(t-t_0))^{2n}}{n^2 p_n}
  =
  \frac{x_0^2}{1-p} +
  \frac{1}{1-p} \sum_{n=1}^\infty \frac{(C^2(t-t_0)^2/p)^n}{n^2},
  \quad
    t\in [t_0, t_0 + \sqrt{p}/C), 
  $$
  in which case the variance is finite. 
  \\
  \noindent $(iii)$ 
            {\em Optimal distribution}.
            Using the Lagrangian
  $$
  \frac{x_0^2}{p_0}
  + \sum_{n=1}^\infty \frac{(Ct)^{2n}}{n^2 p_n}
  + \zeta \left( 1 - \sum_{n=0}^\infty p_n \right) 
$$
  with multiplier $\zeta$,
  we find that the distribution that minimizes the second moment bound
   \eqref{jkfl3} has the form
   \begin{equation}
     \label{djkls} 
   p_0 = c_0 x_0, \qquad 
   p_n = c_0 \frac{(Ct)^n}{n},
   \qquad n\geq 1,
\end{equation} 
   where $c_0 = ( x_0 - \log ( 1 - C t))^{-1}$ is a normalization constant,
   see Figure~\ref{fig3} in which the moment
   bound \eqref{jkfl3} is plotted as a function
   of $C \in [0,\sqrt{p}]$ with $t=1$ for the distribution
   \eqref{djkls} (lower bound) and for the geometric distributions
   with parameters $p = 0.5 , 0.75$, and
   $x_0=1$. 
   
\begin{figure}[H]
\centering
\begin{subfigure}{.49\textwidth}
\centering
\includegraphics[width=\textwidth]{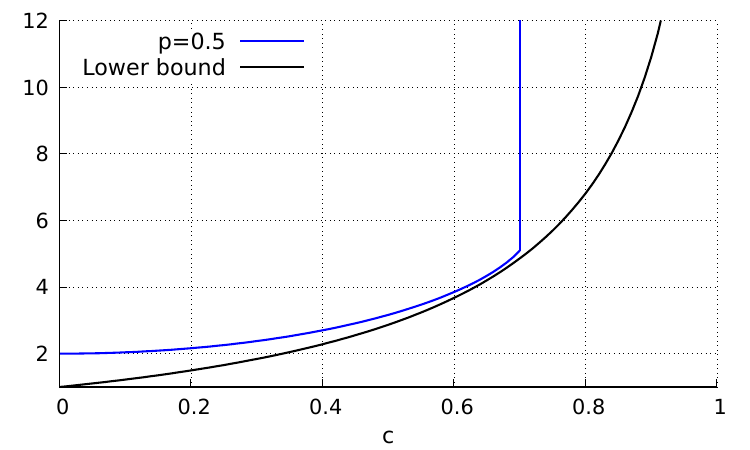}
\vskip-0.1cm
\caption{$p=0.5$.}
\end{subfigure}
\hskip-0.2cm
\begin{subfigure}{.49\textwidth}
\centering
\includegraphics[width=\textwidth]{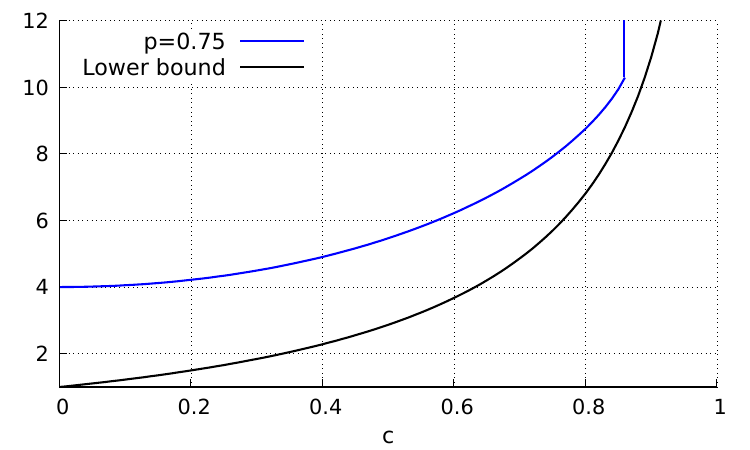}
\vskip-0.1cm
\caption{$p=0.75$.}
\end{subfigure}
\caption{Second moment lower bound.} 
\label{fig3}
\end{figure}

\vskip-0.3cm

\noindent 
 The graphs of Figure~\ref{fig1} are plotted
 using the Poisson and geometric distributions
 with respectively 100,000 and 70,000 Monte Carlo
 samples, in order to match the 22 seconds 
 computation time of Figure~\ref{fig5}-$(a)$ 
 for \eqref{e1}, see Table~\ref{table:example 2}. 
 
\begin{figure}[H]
\centering
\begin{subfigure}{.49\textwidth}
\centering
\includegraphics[width=\textwidth]{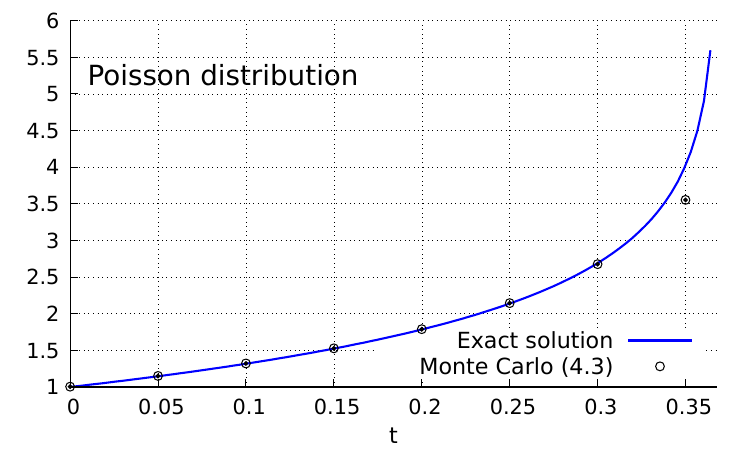}
\vskip-0.1cm
\caption{Poisson tree size.} 
\end{subfigure}
\hskip-0.2cm
\begin{subfigure}{.49\textwidth}
\centering
\includegraphics[width=\textwidth]{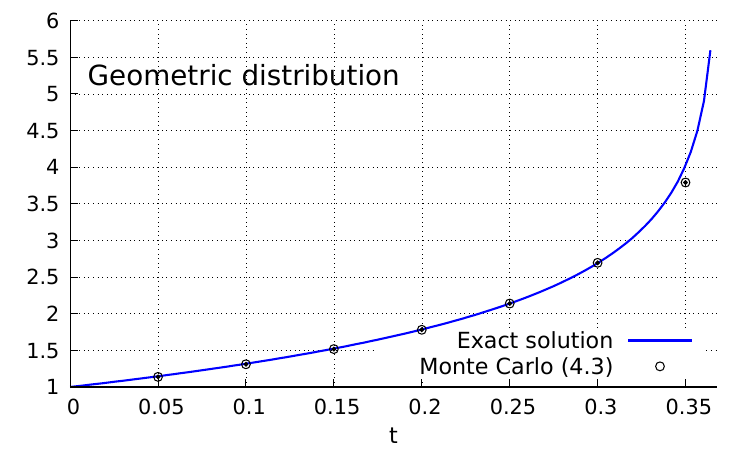}
\vskip-0.1cm
\caption{Geometric tree size.} 
\end{subfigure}
\caption{Numerical solution of \eqref{e1} by the Monte Carlo method \eqref{random-Butcher}.} 
\label{fig1}
\end{figure}

\vskip-0.3cm

\section{Multidimensional codes}
\label{appendix}

\noindent
The next Mathematica code estimates the
Butcher series \eqref{Butcher}
up to a given order $n\geq 1$ in the multidimensional case.
The second component in the output of
\verb|B[f,t,x0,t0,n]| counts 
 the number of trees involved in the Butcher series
truncated up to the order $n\geq 1$. 
  
\smallskip
    
\begin{lstlisting}[language=Mathematica]
  B[f_, t_, x0__, t0_, n_] := (d = Length[x0]; 
   If[n == 0, Return[{x0, 1}], 
    If[n == 1, Return[{x0 + (t - t0)*f[x0], 2}], count = 2; 
     sample = x0 + (t - t0)*f[x0]; 
     g = ConstantArray[Graph[{1 -> 2}], d]; ii = Array[i, n]; 
     For[ii[[1]] = 1, ii[[1]] <= d, ii[[1]]++, 
      g[[ii[[1]]]] = 
       Graph[g[[ii[[1]]]], 
        VertexLabels -> {1 -> D[f[yy], yy[[ii[[1]]]]]}]; 
      g[[ii[[1]]]] = 
       Graph[g[[ii[[1]]]], VertexLabels -> {2 -> f[yy][[ii[[1]]]]}]; 
      m = 1; count += 1; 
      sample += 1/2*(t - t0)^VertexCount[g[[ii[[1]]]]]*
         Product[ff[[2]] , {ff, 
           List @@@ 
            PropertyValue[g[[ii[[1]]]], VertexLabels]}] /. {yy -> 
          x0}]; list = g; 
     While[m <= (n - 2), temp = list; list = {}; 
      Do[l = VertexCount[g]; 
       For[j = 1, j <= l, j++, gg = VertexAdd[g, {l + 1}]; 
        lab = Sort[List @@@ PropertyValue[gg, VertexLabels]][[j]][[
          2]]; For[ii[[l]] = 1, ii[[l]] <= d, ii[[l]]++, 
         gg = Graph[gg, 
           VertexLabels -> {l + 1 -> f[ yy][[ii[[l]]]]}];
         gg = Graph[gg, VertexLabels -> {j -> D[lab, yy[[ii[[l]]]]]}];
          gg = EdgeAdd[gg, j -> l + 1]; 
         GraphPlot[gg, 
          PlotStyle -> {FontSize -> 20, FontColor -> Red}]; 
         count += 1; 
         sample += (t - t0)^(l + 1)/(l + 1)!*
            Product[ff[[2]] , {ff, 
              List @@@ PropertyValue[gg, VertexLabels]}] /. {yy -> 
             x0}; list = Append[list, gg]]], {g, temp}]; m++]; 
     Return[{sample, count}]]]);
x0 = {0, 0.5}; t0 = 0; t1 = 1.3885; 
f[y__] := {1, y[[1]]*y[[2]] + y[[2]]^2}
B[f, t, x0, t0, 4]
\end{lstlisting}
 
\noindent
The next Mathematica code generates a single random Butcher tree sample in \eqref{random-Butcher} for a multidimensional ODE. 
  
\smallskip
    
\begin{lstlisting}[language=Mathematica]
  MCsample[f_, t_, x0_, t0_, dist_] := (d = Length[x0]; 
   n = RandomVariate[dist]; 
   If[n == 0, Return[x0/PDF[dist, 0]], 
    If[n == 1, Return[(t - t0)*f[x0]/PDF[dist, 1]], 
     g = ConstantArray[Graph[{1 -> 2}], d]; ii = Array[i, n]; 
     sample = 0; 
     For[ii[[1]] = 1, ii[[1]] <= d, ii[[1]]++, 
      g[[ii[[1]]]] = 
       Graph[g[[ii[[1]]]], 
        VertexLabels -> {1 -> D[f[yy], yy[[ii[[1]]]]]}]; 
      g[[ii[[1]]]] = 
       Graph[g[[ii[[1]]]], VertexLabels -> {2 -> f[yy][[ii[[1]]]]}]; 
      sample += (t - t0)^2/PDF[dist, 2]/2*
         Product[ff[[2]] , {ff, 
           List @@@ 
            PropertyValue[g[[ii[[1]]]], VertexLabels]}] /. {yy -> 
          x0}]; If[n == 2, Return[sample]]; sample = 0; list = g; 
     m = 1; While[m <= (n - 2), temp = list; list = {}; 
      Do[l = VertexCount[g]; 
       j = RandomVariate[DiscreteUniformDistribution[{1, l}]]; 
       gg = VertexAdd[g, {l + 1}]; 
       lab = Sort[List @@@ PropertyValue[gg, VertexLabels]][[j]][[2]];
        For[ii[[l]] = 1, ii[[l]] <= d, ii[[l]]++, 
        gg = Graph[gg, VertexLabels -> {l + 1 -> f[ yy][[ii[[l]]]]}];
        gg = Graph[gg, VertexLabels -> {j -> D[lab, yy[[ii[[l]]]]]}]; 
        gg = EdgeAdd[gg, j -> l + 1]; 
        GraphPlot[gg, 
         PlotStyle -> {FontSize -> 20, FontColor -> Red}]; 
        If[m == (n - 2), 
         sample += 
          Product[ff[[2]] , {ff, 
             List @@@ PropertyValue[gg, VertexLabels]}] /. {yy -> 
             x0}]; list = Append[list, gg]], {g, temp}]; m++]; 
     Return[sample*(t - t0)^n/PDF[dist, n]/n]]]);
x0 = {0, 0.5}; t0 = 0; t1 = 1.3885; 
f[y__] := {1, y[[1]]*y[[2]] + y[[2]]^2}
MCsample[f, t, x0, t0, GeometricDistribution[0.5]] 
\end{lstlisting}

\footnotesize

\def\cprime{$'$} \def\polhk#1{\setbox0=\hbox{#1}{\ooalign{\hidewidth
  \lower1.5ex\hbox{`}\hidewidth\crcr\unhbox0}}}
  \def\polhk#1{\setbox0=\hbox{#1}{\ooalign{\hidewidth
  \lower1.5ex\hbox{`}\hidewidth\crcr\unhbox0}}} \def\cprime{$'$}

\end{document}